\documentclass[twoside,a4paper]{amsart}

\usepackage[english]{babel}
\usepackage{amssymb,amsmath}
\usepackage{dsfont}

\newtheorem{theorem}{Theorem}[section]

\newtheorem{lemma}[theorem]{Lemma}
\newtheorem{definition}{Definition}[section]
\newtheorem{proposition}[theorem]{Proposition}
\newtheorem{remark}{Remark}[section]

\newtheorem{example}{Example}[section]

\usepackage{color}

\usepackage[all]{xy}

\newcommand{\barr}{\beta}

\newcommand{\ce}{\bar{e}}
\newcommand{\id}{\mathrm{id}}
\newcommand{\red}[1]{\overline{#1}}

\newcommand{\Hom}{\mathrm{Hom}}
\newcommand{\K}{\mathbb{K}}
\newcommand{\Ker}{\mathrm{Ker}}
\newcommand{\Tor}{\mathrm{Tor}}

\author{Ana Paula Santana and Ivan Yudin }
\address{CMUC, Department of Mathematics,
University of Coimbra,
 Coimbra, Portugal}
 \email{aps@mat.uc.pt\qquad yudin@mat.uc.pt}
 \dedicatory{Dedicated to Manuela Sobral}
 \keywords{stratifying ideal, twisted product, relative homological algebra}
\thanks{ This work was partially supported by the Centro de Matem\'atica da
Universidade de Coimbra (CMUC), funded by the European Regional
Development Fund through the program COMPETE and by the Portuguese
Government through the FCT - Funda\c{c}\~ao para a Ci\^encia e a Tecnologia
under the project  PEst-C/MAT/UI0324/2013.  }

\title{Stratifying ideals and twisted products. }

 \def\labelauthor{APSIY} 

\begin{document}

\begin{abstract} We study stratifying ideals for rings in the context of relative homological algebra.
Using $LU$-decompositions, which are
a special type of  twisted products, we give a sufficient condition for an idempotent ideal  to be (relative) stratifying.

\end{abstract}
	\maketitle

	\section{Introduction}\label{\labelauthor:Introduction}
	\setcounter{equation}{0}
The notion of stratifying ideal was introduced, almost simultaneously, in several
articles, although under different names. The first reference we could find to these ideals is
 \cite{\labelauthor:bib:cps88}, where they are considered, without any proper designation, in the context of quasi-hereditary
algebras. Then they were studied in \cite{\labelauthor:bib:auslander} under the name of
\emph{strong idempotent ideals}. Almost simultaneously, in \cite{\labelauthor:bib:geigle}, the notion of
homological epimorphism was introduced. Stratifying ideals are exactly the
kernels of surjective homological epimorphisms of rings.  The term \emph{stratifying ideal} seems
to appear for the first time in \cite{\labelauthor:bib:cps96}.

Our interest in stratifying ideals
was motivated by the problem of constructing minimal projective resolutions. In fact, let $\Lambda$
be a finite dimensional algebra over a field. One
of the many ways to define a stratifying ideal of $\Lambda$ is the following. Given an idempotent $e$  in $\Lambda$, denote by $J$ the
ideal $\Lambda e \Lambda$ and by
$\red{\Lambda}$ the quotient
$\left.\raisebox{0.3ex}{$\Lambda$}\middle/\!\raisebox{-0.3ex}{$J$}\right. $.
The ideal $J$ is stratifying if, for any $\red{\Lambda}$-module $M$ and
projective resolution $P_\bullet \to M$ of $M$  over $\Lambda$, the complex
$\red{\Lambda}\otimes_{\Lambda} P_\bullet \to M$ is a projective resolution of $M$ over
$\red{\Lambda}$. Moreover, if $P_\bullet\to M$ is minimal, then the same is true for
$\red{\Lambda}\otimes_{\Lambda} P_\bullet \to M$.
Therefore one can construct minimal projective resolutions over $\red{\Lambda}$ by
constructing them first over $\Lambda$ and then applying the functor $\red{\Lambda}\otimes_\Lambda
-$.

It is usually quite difficult to verify if a given ideal of $\Lambda$ is
stratifying. It is well known that hereditary ideals are
stratifying. More generally, idempotent ideals of $\Lambda$ which are projective left $\Lambda$-modules are stratifying.

The aim of this paper is to give a new sufficient condition, Theorem~\ref{\labelauthor:teo:main}, for an idempotent ideal $\Lambda e \Lambda$ to be stratifying.
This result will be used in our work on homological properties of (quantised)
Schur algebras (see \cite{\labelauthor:bib:advances, \labelauthor:bib:quantised}).

The paper is organized as follows. In Section~\ref{\labelauthor:hom} we give a short overview
of relative homological algebra over rings with identity, and  define the bar resolution and \emph{relative} stratifying ideals in this context.
We also relate our definition of relative stratifying ideal to the usual definition of stratifying ideal of a
 finite dimensional algebra over a field.

In the first part of Section~\ref{\labelauthor:twisted} we define twisted products for
relative pairs.  Proposition~\ref{\labelauthor:pro:indbar}  relates twisted products and bar
resolutions. The second part of Section~\ref{\labelauthor:twisted} is dedicated to
$LU$-decompositions of a ring $A$ with a fixed idempotent $e$. In
Theorem~\ref{\labelauthor:teo:main}  we prove that $AeA$
is a relative stratifying ideal if $A$ admits an $LU$-decomposition.

\section{Relative homological algebra}\label{\labelauthor:hom}
\setcounter{equation}{0}

In this section we recall the  definitions and  results in
relative homological theory that we will use in the article. All these
notions and results are given in terms of left modules, but they can also
be applied to right modules, with appropriate changes in formulae if
necessary. By an $A$-module we mean a  left $A$-module and we write $A$-mod  for the category of left $A$-modules.

Relative homological theory was originally developed in \cite{\labelauthor:bib:hochschild} and we follow the terminology used in this work. A
more detailed (and slightly more general) treatment of this topic can be found in Chapter VIII
of \cite{\labelauthor:bib:homology}.

Let $A$ be a ring with identity $1$ and $S$ a subring of $A$ containing
$1$. We will refer to  $(A,S)$ as a
\emph{relative pair}.
An exact sequence of
$A$-modules will be called \emph{$(A,S)$-exact} if the kernel of every differential is
an $S$-direct summand of the corresponding object.
Equivalently a complex of $A$-modules
\begin{equation*}
	\dots \rightarrow M_{k} \xrightarrow{d_k}  M_{k-1} \rightarrow \dots
\end{equation*}
is $(A,S)$-exact if there are $S$-homomorphisms $s_k\colon M_{k} \to M_{k+1}$
such that $d_{k+1} s_k + s_{k-1} d_k = \id_{M_k}$ for all meaningful values of  $k$.

We say that an $A$-module $P$ is  \emph{$(A,S)$-projective} if for every short
$(A,S)$-exact sequence
\begin{equation}
	\label{\labelauthor::eq:ses}
	0 \rightarrow X \rightarrow M \xrightarrow{f} N \rightarrow 0
\end{equation}
and every $A$-homomorphism $g\colon P\to N$ there is an $A$-homomorphism
${h\colon P\to M} $ such that the diagram
\begin{equation*}
	\xymatrix{
	& & & P \ar[ld]_-{h} \ar[d]^-{g}\\
	 0 \ar[r] & X \ar[r] & M \ar[r]^-{f} & N \ar[r] & 0
	}
\end{equation*}
commutes.
In other words, $P$ is  $(A,S)$-projective if for every short {$(A,S)$-exact
sequence} \eqref{\labelauthor::eq:ses}, the map
\begin{equation*}
	\Hom_{A}(P,f) \colon \Hom_{A}(P,M) \to \Hom_{A}(P,N)
\end{equation*}
	is surjective.

\begin{remark}
	\label{\labelauthor::rem:semisimple}
	Obviously every projective $A$-module is $(A,S)$-projective.
	On the other hand,
if $S$ is a semisimple ring, then every
$(A,S)$-projective module is projective. In fact, in this case, every exact
sequence is automatically $(A,S)$-exact. Therefore  the condition for an $A$-module to be
$(A,S)$-projective coincides with the condition for it to be projective.
\end{remark}

\begin{example}
	\label{\labelauthor:ex:free}
	Let $V$ be an $S$-module. Then by Lemma~2 in \cite{\labelauthor:bib:hochschild} and subsequent considerations, the
	$A$-module $A\otimes_S V$ is $(A,S)$-projective. Moreover, an $A$-module $M$  is $(A,S)$-projective if and only if it is isomorphic to a
	direct $A$-module summand of  $A\otimes_S V$, for some $S$-module $V$.
Modules of the form $A\otimes_S V$ will be called
	\emph{$(A,S)$-free}.
\end{example}

It is interesting to note that $(A,S)$-projective modules behave well under change of base rings.
\begin{lemma}
	\label{\labelauthor:lem:proj}
Let $(A,S)$ and $(R,D)$ be relative pairs, and $\phi\colon A\to R$ a
homomorphism of rings such that $\phi(S)\subset D$ and $\phi(1)=1$. Suppose $P$ is an $(A,S)$-projective module.
Then $R\otimes_A P$ is an $(R,D)$-projective module.
\end{lemma}
\begin{proof}
	We know from Example~\ref{\labelauthor:ex:free} that $P$ is isomorphic to a direct summand of
	$A\otimes_S V$, for some $S$-module $V$. Since the functor $R\otimes_A -$
	is additive, the $R$-module $R\otimes_A P$ is isomorphic to a direct summand of
	the free $(R,D)$-module
	\begin{equation*}
	R\otimes_A A \otimes_S V \cong R\otimes_S V \cong R\otimes_D D \otimes_S
	V.
	\end{equation*}
	This shows that $R\otimes_A P$ is $(R,D)$-projective.
\end{proof}
An $(A,S)$-exact sequence of left $A$-modules
\begin{equation*}
	\dots \rightarrow P_k \xrightarrow{d_k} P_{k-1} \rightarrow \dots
	\rightarrow P_1 \xrightarrow{d_1} P_0 \xrightarrow{d_0} M \rightarrow
	0
\end{equation*}
is called an \emph{$(A,S)$-projective resolution} of $M\in A$-mod if each $P_k$ is an $(A,S)$-projective  module.

Next we describe the bar resolution $B(A,S,M)$ of $M\in A$-mod. This construction will provide an $(A,S)$-projective resolution for $M$.
We set
\begin{equation*}
	B_{-1}\left(
	A, S,M \right)= M  \,\,\, \mbox{and}\,\,\,  B_k\left(A, S,M \right) =
	A^{\otimes_S^{(k+1)}
	} \otimes_S M, \,\,\, \mbox{all}\,\,\,k\ge 1,
\end{equation*}
where $A^{\otimes_S^l}$ stands for the $l$th tensor power of $A$ over $S$.
Now we define $A$-module homomorphisms $d_{kj}\colon B_{k}\left( A,S,M
\right)\to B_{k-1}\left( A,S,M \right) $, $0\le j\le k$,
and $S$-module homomorphisms $s_k\colon B_k\left( A,S,M \right)\to B_{k+1}\left(
A,S,M \right)$ by
\begin{align*}
	d_{0,0} \left( a\otimes m \right) & = am,\\
	d_{k,0} \left( a_0\otimes a_1 \otimes \dots \otimes a_k\otimes m \right)
	&= a_0 a_1 \otimes a_2 \otimes \dots \otimes a_k \otimes
	m,\\
	d_{k,j} \left( a_0\otimes a_1 \otimes \dots \otimes a_k\otimes m \right)
	& =  a_0\otimes a_1 \otimes \dots \otimes a_j a_{j+1} \otimes \dots \otimes a_k \otimes m,
	\\&\makebox[15em]{}	 1\le j\le k-1,\\
	d_{k,k} \left( a_0\otimes a_1 \otimes \dots \otimes a_k\otimes m \right)
	&= a_0\otimes a_1\otimes \dots \otimes a_{k-1} \otimes a_k m,\\
	s_{-1} \left( m \right) & = 1\otimes m,\\
	s_k \left( a_0\otimes a_1 \otimes \dots \otimes a_k \otimes m \right) &=
	1 \otimes   a_0  \otimes a_1 \otimes
	\dots \otimes a_k \otimes m,\  0 \le k.
\end{align*}
Define $d_k\colon B_k(A,S,M) \to B_{k-1}(A,S,M)$ by $d_k = \sum_{t=0}^k
(-1)^t d_{k,t}$. Then one can verify (cf. \cite[Section~3]{\labelauthor:bib:advances}) that
\begin{align}
	\label{\labelauthor:eq:contraction0}
		d_0 s_{-1} & = \id_{B_{-1}(A,S,M)},\\
		\label{\labelauthor:eq:contraction}
		d_{k+1} s_k + s_{k-1} d_k & = \id_{B_{k}(A,S,M)},\ k\ge 0,\\
		\label{\labelauthor:eq:differential}
		d_{k} d_{k+1}& =0, \ k\ge 0.
\end{align}
We have the following result.
\begin{proposition}
	\label{\labelauthor:pro:bar}
	Let $(A,S)$ be a relative pair and $M$  a left $A$-module.
	Then the  complex $B(A,S,M)=(B_{k}(A,S,M), d_k)_{k\ge -1}$ is an $(A,S)$-projective
	resolution of $M$.
\end{proposition}
\begin{proof}
	From Example~\ref{\labelauthor:ex:free}, we know that all the modules $B_k(A,S,M)$, for
	$k\ge 0$, are $(A,S)$-projective. From~\eqref{\labelauthor:eq:differential}, it follows
	that $B(A,S,M)$ is a complex. 
	From~\eqref{\labelauthor:eq:contraction0}
	and
	\eqref{\labelauthor:eq:contraction}, we get 
	that $B(A,S,M)$ is contractible as a complex of left $S$-modules.
\end{proof}
The $(A,S)$-projective resolution $B(A,S,M)$  is called the
\emph{bar resolution} for~$M$.
We will write $\barr (A,S,M)$ for the complex obtained from $B(A,S,M)$ by
deleting the term $B_{-1}(A,S,M)$.

The bar resolution for  a right $A$-module $N$ is defined in a similar way to
the one described above. It will be denoted by $\red{B} (N,S,A)$.  In
~\cite[Corollary IX.8.2]{\labelauthor:bib:homology}, there it is proved that $B(A,S,A) \cong
\overline{B}(A,S,A)$ and
\begin{equation}\label{\labelauthor:eq:rightleft}
B(A,S,M) \cong B(A,S,A) \otimes_A M \,\,\, \mbox{and}\,\,\,  \red{B} (N,S,A) \cong N\otimes_A B(A,S,A).
\end{equation}

Using the bar resolution, it is possible to define relative Tor-groups
(cf.~\cite[(IX.8.5)]{\labelauthor:bib:homology}).
Suppose we are given a  left $A$-module $M $  and a right $A$-module $N$. Then
the   \emph{relative $\Tor$-groups} are defined as

\begin{equation*}
	\Tor^{(A,S)}_k(N,M) = H_k \left( N\otimes_A \barr(A,S,A)\otimes_A M \right).
\end{equation*}
Suppose $P_\bullet\to M\to 0$ and $Q_\bullet \to N\to 0$ are $(A,S)$-projective
resolutions. Then, by Theorem~IX.8.5  in \cite{\labelauthor:bib:homology}, we have
\begin{equation}
	\label{\labelauthor:eq:toresolution}
	\Tor^{(A,S)}_k (N,M) \cong H_k \left( N \otimes_A P_\bullet \right)
	\cong H_k\left(  Q_\bullet\otimes_A M \right).
\end{equation}
\begin{remark}
	\label{\labelauthor:rem:Torsemisimple} In case $S$ is a semisimple ring, we have $	\Tor^{(A,S)}_k (N,M) \cong 	\Tor^A_k (N,M)$ for all $k\geq 0$. In fact, for $S$ semisimple, $P_\bullet\to M\to 0$ is an $(A,S)$-projective
resolution of $M$ if and only if  it is a projective resolution of $M$ as
$A$-module (see Remark~\ref{\labelauthor::rem:semisimple}).
\end{remark}
It is now possible to introduce the notion of
$(A,S)$-stratifying ideal.
Given a relative pair $(A,S)$  and an  idempotent $e\in S$, we write $\red{A} :=
A/AeA$.

\begin{definition}
	The ideal $AeA$ is called  \emph{$(A,S)$-stratifying} if
	${\Tor^{(A,S)}_k (\red{A},\red{A})=0}$, for all $k \geq 1$.
\end{definition}
This definition of $(A,S)$-stratifying ideal is closely connected with the definition of stratifying ideal given in \cite{\labelauthor:bib:cps96}. In fact, in many situations they are equivalent.

\begin{proposition}
	\label{\labelauthor:pro:ttt}
	Let $(A,S)$ be a relative pair with $A$ a finite dimensional
	algebra over a field. Suppose that  $S$ is  a semisimple algebra and  $e\in S$ is an idempotent. Then  the following conditions are equivalent.
	\begin{enumerate}
	\item The ideal  $AeA$ is
	$(A,S)$-stratifying;	\item  $AeA$ is a strong idempotent ideal in the sense of
			\cite{\labelauthor:bib:auslander};
		\item $A\to \red{A}$ is a homological epimorphism in the sense
			of \cite{\labelauthor:bib:geigle};
		\item $AeA$ is a stratifying ideal in the sense of
			\cite{\labelauthor:bib:cps96}.
	\end{enumerate}
\end{proposition}
\begin{proof}
If $S$ is semisimple, we know, from  Remark~\ref{\labelauthor:rem:Torsemisimple}, that $	\Tor^{(A,S)}_k (\red{A},\red{A}) \cong 	\Tor^A_k (\red{A},\red{A}), \,\,\,k\geq 0$.
Therefore, by Proposition~1.3(iv') in \cite{\labelauthor:bib:auslander}, the ideal $AeA$ is
	$(A,S)$-stratifying if and only if it is strong
idempotent in the sense of \cite{\labelauthor:bib:auslander}.

Since $A\to \red{A}$ is an epimorphism of rings, the multiplication map $\red{A} \otimes_A \red{A}\to \red{A}$ is an isomorphism. Then, from
 Theorem~4.4(1) in~\cite{\labelauthor:bib:geigle}, it follows that $A\to \red{A}$ is a
homological epimorphism if and only if $AeA$ is
	$(A,S)$-stratifying.

Finally, from Theorem~4.4(5') in \cite{\labelauthor:bib:geigle} and Remark~2.1.2(a)
in~\cite{\labelauthor:bib:cps96}, we get that $AeA$ is a stratifying ideal in the sense of
\cite{\labelauthor:bib:cps96} if and only if $A\to \red{A}$ is a
homological epimorphism.
\end{proof}

\section{Twisted products} \label{\labelauthor:twisted}
\setcounter{equation}{0}

\subsection{General definitions and bar resolution}
In this section we introduce the notion of a twisted product for relative pairs
and discuss bar resolutions in this setting.

\begin{definition}
	Let $(A,S)$ be a relative pair and  $A_1,$  $A_2$ subrings of
	$A$ containing $S$. We say that $(A,S)$ is  a \emph{twisted product} of
	$A_1$ and $A_2$ if the map $\alpha \colon A_1\otimes_S A_2\to A$ induced
	by the multiplication in $A$ is an isomorphism of abelian groups.
\end{definition}
If $(A,S)$ is a twisted product of $A_1$ and $A_2$, we can define the twisting map
$T\colon A_2 \otimes_S A_1 \to A_1 \otimes_S A_2$ as  the composition
\begin{equation*}
	A_2 \otimes_S A_1 \xrightarrow{\mu_A} A \xrightarrow{\alpha^{-1}} A_1\otimes_S A_2,
\end{equation*}
where $\mu_A$ is the multiplication in $A$.
The existence of this map motivated the name ``twisted product''.
\begin{example}
	Let $G$ be a group with identity $e_G$ and $H_1$, $H_2$ subgroups of $G.$ Suppose  that $H_1\cap H_2
	= \left\{ e_G \right\}$ and $H_1 H_2 = G$. Then one says that $G$ is a
	\emph{Zappa-Sz\'ep product} of $H_1$ and $H_2$. Given any commutative ring  with identity $S$, it can be  checked that  the relative pair $(SG,S)$ is a
	twisted product of $SH_1$ and $SH_2$.
	The Zappa-Sz\'ep product was developed independently by Zappa
	in~\cite{\labelauthor:bib:zappa} and Sz\'ep in~\cite{\labelauthor:bib:szep}.
\end{example}
	Suppose a relative pair $(A,S)$ is a twisted product of subrings $A_1$ and $A_2
	$.   Then the endofunctors $A\otimes_S -$, $A_1\otimes_S -$, and $A_2\otimes_S
	-$ on the category of $S$-mod can be turned into  monads using
	multiplication and units of algebras in the obvious way. 
Moreover, $T$
induces a natural transformation $\tau$ between the functors $A_2 \otimes_S A_1
\otimes_S - \to A_1\otimes_S A_2\otimes_S -$. One can check that $\tau$ is a
distributive law in the sense of \cite{\labelauthor:bib:beck}.
Twisted products of algebras were also studied in~\cite{\labelauthor:bib:cap}.

Suppose $(A,S)$ is a twisted product of subrings $A_1$ and $A_2$. Then we can consider
every $A$-module  as an $A_1$-module and every $A_2$-module as an
$S$-module. Thus, we have two functors $A\otimes_{A_2} -$ and $A_1\otimes_S -$
 from $A_2$-mod to $A_1$-mod.

\begin{lemma}
	\label{\labelauthor:lem:indiso}
	Suppose $(A,S)$ is a twisted product of subrings $A_1$ and $A_2$. Then the functors $A\otimes_{A_2} -$, $A_1\otimes_S -\colon
	A_2\mbox{-\emph{mod}} \to A_1\mbox{-\emph{mod}}$ are isomorphic.
\end{lemma}
\begin{proof}
Given any $A_2$-module $M$ we define $f_M$ as  the composition
of the three $A_1$-isomorphisms natural in $M$
\begin{equation}
	\label{\labelauthor:eq:indiso}
	 A_1\otimes_S M \xrightarrow{\cong} A_1\otimes_S \left(A_2 \otimes_{A_2}
	M\right) \xrightarrow{\cong} \left( A_1\otimes_S A_2 \right)
	\otimes_{A_2} M
	\xrightarrow{\alpha\otimes_{A_2} M} A\otimes_{A_2} M.
\end{equation}
Then $f:=(f_M)_{M\in A_2\mbox{-mod}}$ is the required isomorphism of fuctors.
\end{proof}

\begin{proposition}
	\label{\labelauthor:pro:indbar}
	Suppose that the relative pair $(A,S)$ is a twisted product of $A_1$ and
	$A_2$. Then for any $A_2$-module $M$ the complex $A\otimes_{A_2}
	B(A_2,S,M)$ is an $(A,S)$-projective resolution of $A\otimes_{A_2} M$.
\end{proposition}
\begin{proof}

For $k\ge 0$
we have isomorphisms of $A$-modules
\begin{equation*}
	A\otimes_{A_2} A_2^{\otimes_S ^{(k+1)}} \otimes_{S} M \cong A \otimes_S
	A_2^{\otimes_{S}^ k} \otimes_S M.
\end{equation*}
Therefore, by Example~\ref{\labelauthor:ex:free}, the modules $A\otimes_{A_2} B_k(A_2,S,M)$ are
$(A,S)$-projective for all $k\ge 0$.

By Lemma~\ref{\labelauthor:lem:indiso}, the functors $A\otimes_{A_2} -$ and $A_1\otimes_S -$
are isomorphic as functors from the category of $A_2$-modules to the category of $A_1$-modules, and so they are
also isomorphic as functors to the category of $S$-modules.
Therefore, to show that $A\otimes_{A_2} B(A_2,S,M)$ is splittable as a complex of
$S$-modules, it is enough to show that $A_1\otimes_S B(A_2,S,M)$ is splittable as
a complex of $S$-modules.
Since $\left.A_1\otimes_S-\right.$ is an additive endofunctor in the category of $S$-modules,
and $B(A_2,S,M)$ is a splittable exact sequence in this category,
we get that  ${A_1\otimes_S B(A_2,S,M)}$ is a splittable exact sequence of
$S$-modules.
This shows that $A\otimes_{A_2} B(A_2,S,M)$ is an $(A,S)$-projective resolution of
$A\otimes_{A_2} M$.
\end{proof}

\subsection{LU-twisted products and $(A,S)$-stratifying ideals}
Let $(A,S)$ be a relative pair and $e\in S$ an idempotent. We denote by
$\ce$ the idempotent $1-e$.
Given a subring $B$ of $A$ containing $S$ it is convenient to think of
$B$ as the matrix ring
\begin{equation*}
	\left(\begin{array}{cc}eBe & eB\ce \\ \ce B e & \ce B
		\ce\end{array}\right).
	\end{equation*}
	Note that $( eBe,e)$ and $(\ce B \ce,\ce)$ are rings.
We will say that $B$ is \emph{upper triangular} if $\ce B e = 0$, \emph{lower
triangular} if $e B\ce = 0$, and \emph{diagonal} if $\ce B e = e B \ce =0$.

We will write $\red{B}$ for the quotient of $B$ by the ideal $BeB$.
\begin{proposition}
	\label{\labelauthor:pro:upper}
Let $(A,S)$ be a relative pair and $e\in S$ an idempotent.
Suppose that $B$ is an upper or lower triangular subring of $A$. Then
$\red{B} \cong \ce B \ce$, where the isomorphism is induced by
the inclusion of $\ce B\ce $ into $B$.
\end{proposition}
\begin{proof} We prove the proposition in the case when $B$ is an upper triangular ring. The lower triangular case is similar.

Let $b\in B$. Then $b = (\ce+e) b (\ce + e) = \ce b \ce + e be +e b \ce$. This
shows that $[b] = [\ce b \ce]$ in $\red{B}$. Thus the map $\phi\colon\ce B \ce \to \red{B}$,
$b \mapsto [b]$ is a surjective ring homomorphism.
To check that $\Ker(\phi) = 0$ it is enough to notice that
$BeB \cap \ce B \ce \subset \ce B e B \ce =0,$
since $\ce B e = 0$.
\end{proof}

\begin{definition}
	Let $(A,S)$ be a relative pair with 	$S$  diagonal and $e\in S$ an idempotent. We say that
	$A$ admits an $LU$-decomposition if there are subrings $L$ and $U$ of
	$A$ containing $S$ such that:
	\begin{enumerate}
		\item $(A,S)$ is a twisted product of $L$ and $U$;
		\item $L$ is lower triangular;
		\item $U$ is upper triangular.
	\end{enumerate}
\end{definition}
Before we state and prove the main theorem,  we  need two technical results. Their proofs use the following proposition, which can be found in
\cite{\labelauthor:bib:homology}.

\begin{proposition}
	\label{\labelauthor:pro:Mac}\rm{(IX.9.3
\cite{\labelauthor:bib:homology})} Suppose that the ring $R$ is the direct product of two subrings $R_1$ and $R_2$. Given a right $R$-module $N$ and a left $R$-module $M$, there is an isomorphism {of abelian groups}
\begin{equation*}
	N \otimes_R M \cong ( N \otimes_R R_1) \otimes_{R_1}( R_1\otimes_R M) \oplus ( N \otimes_R R_2) \otimes_{R_2}( R_2\otimes_R M).
\end{equation*}
\end{proposition}

\begin{proposition}
	\label{\labelauthor:pro:LU}
	Let $(A,S)$ be a relative pair with 	$S$  diagonal and $e\in S$ an idempotent. Suppose that
	$(A,S)$ admits an $LU$-decomposition with subrings $L$ and $U$.
	Then $\red{A}$
	is a twisted product of $\red{L}$ and $\red{U}$.
\end{proposition}
\begin{proof}
To show that $\red{A}$  is a twisted product of $\red{L}$ and $\red{U}$
	we first need to prove that
$\red{S}$ can be considered a subring of  $\red{L}$ and $\red{U}$, and $\red{L}$
and $\red{U}$  can be considered as  subrings
	of $\red{A}$. For this it is enough  to verify that $L \cap AeA = LeL,$
	$U\cap AeA = UeU$ and $S\cap AeA= SeS.$

Using the fact that $eL\ce= \ce U e=0$, we have
\begin{align*}
AeA& = LU e LU = LeUe LU + L\ce U e LU = LeUeLeU + LeUeL\ce U \subset Le A e U
\\& = LeLUeU = LeL e UeU + LeL \ce U e U = LeL e UeU \subset LeU.
\end{align*}
Thus $AeA = LeU$.
Since $S$ is diagonal, it is the direct product of the rings $(eSe,e)$ and $\left(
\ce S \ce,\ce \right)$ and we have an isomorphism of abelian groups (cf.~Propo\-si\-ti\-on  \ref{\labelauthor:pro:Mac})
\begin{equation}\label{\labelauthor:eq:directsum}\begin{array}{ccccc}
		
 \gamma \colon Le \otimes_{eSe}eU \oplus L\ce \otimes_{\ce S\ce} \ce U & \rightarrow &	L\otimes_S U & \xrightarrow{\alpha} &  A\\
 (a\otimes b, a' \otimes b') & \mapsto & (a+a')\otimes( b+b') & \mapsto & ab+a'b'\,.\end{array}
\end{equation}
Therefore
$LeU \cap L\ce U = \gamma \left(Le \otimes_{eSe}eU \right) \cap  \gamma\left(
L\ce \otimes_{\ce S\ce} \ce U \right)= 0$ and so ${Le \cap L\ce U =0}$. Since $L = Le \oplus L\ce$
and $Le \subset LeU$,
this implies $L\cap LeU = Le \cap LeU = Le$. Now
\begin{equation*}
	Le \subset LeL = LeLe \oplus LeL \ce = LeLe \subset Le.
\end{equation*}
Therefore
$
L \cap AeA = L \cap LeU = Le = LeL,
$
as required.

In a similar way it can be proved $U \cap AeA = UeU$. To show that $S \cap AeA=SeS$ it is enough to notice that
\begin{equation*}
S \cap AeA=  S \cap L \cap AeA= S\cap Le=Se\subset SeS= SeSe \oplus SeS\ce=SeSe \subset Se.
\end{equation*}
Thus $S \cap AeA= SeS.$

Next we have to check that the map
\begin{align*}
	\red{\alpha}\colon& \red{L} \otimes_{\red{S}} \red{U} \to \red{A}\\
&[l] \otimes [u] \mapsto  [lu]
\end{align*}
is an isomorphism. By Proposition~\ref{\labelauthor:pro:upper}, we know that $\red{L} \cong \ce L\ce$,
$\red{U}\cong \ce U \ce$ and $\red{S} \cong \ce S\ce$. Therefore, we can replace
 $\red{\alpha}$ by the map
\begin{align*}
	\beta\colon \ce L \ce \otimes_{\ce S \ce } \ce U \ce & \to
	\red{A}\\
	l \otimes u & \mapsto [lu].
\end{align*}
Notice that $L\ce= \ce L \ce \oplus eL\ce = \ce L \ce$ and $\ce U=\ce U \ce \oplus \ce U e = \ce U \ce$. Therefore $\beta $ can be decomposed in the following way:
\begin{equation*}
	\xymatrix{
	L\ce \otimes_{\ce S \ce} \ce U  \ar[r]^-{\beta} \ar[d]_-{\imath} &
	\red{A}\\
	Le \otimes_{e S e} e U \oplus L\ce \otimes_{\ce S \ce} \ce U
	\ar[r]_-{\gamma}^-{\cong} & A \ar[u]_-{\pi}
	}
\end{equation*}
Now $\Ker (\pi)= AeA=LeU = \gamma \left(Le \otimes_{eSe} eU \right),$ which implies $\Ker (\pi \gamma) =Le\otimes_{eSe}eU $. Therefore $\Ker ( \beta)=0$ and $ \beta \left( L\ce\otimes_{\ce S\ce}\ce U\right)=\red{A} $.
\end{proof}

\begin{proposition}
	\label{\labelauthor:pro:reduction}
Let $(A,S)$ be a relative pair with $S$ diagonal and $e\in S$ an idempotent. Suppose that $A$ admits an
$LU$-decomposition with subrings $L$ and~$U$. Then $A\otimes_{U} \red{U} \cong
\red{A}$ as $A$-modules and $\red{L} \otimes_L A \cong \red{A}$ as right
$A$-modules.	
\end{proposition}
\begin{proof}
	Using Proposition~\ref{\labelauthor:pro:Mac}, we have
	\begin{equation*}
		L\otimes_S  \red{U} \cong Le \otimes_{eSe} e \red{U} \oplus L\ce
		\otimes_{\ce S \ce} \ce \red{U} = L\ce \otimes_{\ce S \ce }
		\red{U}.
	\end{equation*}
As $L$ is lower triangular, we know that $L\ce = \ce L\ce$. Also $\ce U \ce \to
\red{U}$, $u \mapsto [u]$ is an isomorphism by Proposition~\ref{\labelauthor:pro:upper}.
Therefore,
	the map
	\begin{align*}
		\phi\colon \ce L\ce \otimes_{\ce S \ce } \ce U \ce &\to L
		\otimes_S \red{U}\\
		l \otimes u & \mapsto l \otimes [u],
	\end{align*}
	is an isomorphism. We remind the reader that in the proof of
	Lemma~\ref{\labelauthor:lem:indiso} we constructed the isomorphism
	$f_{\red{U}} \colon L\otimes_S \red{U} \to A \otimes_{U} \red{U}$, given
	 by $l\otimes [u] \mapsto l \otimes [u]$. We write $\psi:=f_{\red{U}} \phi$.
Consider the isomorphism
	$\beta\colon \ce L\ce \otimes_{\ce S \ce} \ce U \ce\to \red{A}$ constructed in the proof of Proposition~\ref{\labelauthor:pro:LU} . Then we have the isomorphisms
	\begin{equation*}
		\xymatrix{ &\ce L\ce \otimes_{\ce S \ce} \ce U \ce\ar[ld]_{\beta}
		\ar[rd]^{\psi} & && l\otimes u \ar@{|->}[ld] \ar@{|->}[rd] \\
		\red{A} && A\otimes_U \red{U} & [lu] && l\otimes [u]}
	\end{equation*}
	of abelian groups. 
	Write $\tau := \beta
	\psi^{-1}\colon A\otimes_{U} \red{U} \to \red{A}.$ It is our aim to prove
	that $\tau$ is a homomorphism of $A$-modules. It is obvious that
	$\tau$ is a homomorphism of $L$-modules. Thus, as $A$ is the twisted product of $L$ and $U$, to prove that $\tau$ is a homomorphism of $A$-modules it is enough
	to show that $\tau$ is a homomorphism of $U$-modules. For this, let
	$u'\in U$ and $l\otimes [u]\in A\otimes_{U} \red{U}$. Then, as $u'l \in A$, we have $u'l =\alpha\left(\sum_{i\in I} l_i\otimes  u_i\right)=
	\sum_{i\in I} l_i u_i$, for some
	 finite set $I$, $l_i\in L$ and $u_i \in U$. Thus
	\begin{align*}
	\tau (u'( l\otimes [u]))& = \tau \left( \sum_{i\in I} l_i u_i \otimes
	[u]\right)= \sum_{i \in I} \tau \left(l_i \otimes [u_i u]\right) \\[2ex]&=
	\sum_{i\in I} [l_i u_i u]=
	 \left[\left( \sum_{i\in I}l_i u_i\right)
	u\right] =  [u'lu]=u'	\tau (l \otimes [u] ).
	\end{align*}
	Therefore $A\otimes_U \red{U} \cong \red{A}$ as $A$-modules. Applying
	this result to the opposite algebras, we conclude that $\red{L} \otimes_L
	A \cong \red{A}$ as right $A$-modules.
\end{proof}

We  are now ready to prove the main result of the article.
\begin{theorem}
	\label{\labelauthor:teo:main}
Let $(A,S)$ be a relative pair with $S$ diagonal and $e\in S$ an idempotent.
Suppose that  $(A,S)$ admits an $LU$-decomposition with subrings $L$ and $U$. Then  $AeA$ is
an $(A,S)$-stratifying ideal.
\end{theorem}
\begin{proof}
	We can consider $\red{U}$ as a $U$-module and so, by
	Proposition~\ref{\labelauthor:pro:indbar}, the complex $A\otimes_U B(U,S,\red{U})$ is an
	$(A,S)$-projective resolution of $A \otimes_{U}\red{U}$. We know, from
	Proposition~\ref{\labelauthor:pro:reduction}, that $A\otimes_{U} \red{U}\cong
	\red{A}$ as $A$-modules. Therefore, $A\otimes_U B(U,S,\red{U})$ gives an
	$(A,S)$-projective resolution of $\red{A}$.

	If we show that  the complex $\red{A}\otimes_A A \otimes_U
	B(U,S,\red{U})$ is exact, then $$\Tor^{(A,S)}_k(\red{A}, \red{A}) {= 0},$$  for all $k \geq 1$, and $AeA$ is
an $(A,S)$-stratifying ideal.
We have an obvious isomorphism of complexes
	\begin{equation*}	
		\red{A}\otimes_A A \otimes_U
		B(U,S,\red{U}) \cong \red{A} \otimes_{U} B(U,S,\red{U}).
	\end{equation*}
	Therefore to prove the theorem it is enough to check that $\red{A} \otimes_U B(U,S,\red{U})$ is
	exact.
Consider the maps
\begin{align*}
	\phi_{-1} \colon \red{A} \otimes_{U} \red{U} & \to \red{A}
	\otimes_{\red{U}} \red{U}&
	\\[2ex]
	[a] \otimes [u] & \mapsto [a] \otimes [u] &
\\[2ex]
\phi_k\colon \red{A} {\otimes}_U U^{{\otimes}_S^{(k+1)}
	} \otimes_S
	\red{U} & \to \red{A}\otimes_{\red{U}}
	\red{U}^{{\otimes}_{\red{S}}^{(k+1)} } \otimes_{\red{S}} \red{U}\\[2ex]
[a] \otimes u_0 \otimes \dots\otimes u_k \otimes [u]
		&\mapsto [a] \otimes [u_0] \otimes \dots \otimes [u_k] \otimes
		[u],\qquad k\ge 0.
\end{align*}
It is straightforward to verify that $\phi:=(\phi_k)_{k\ge -1}$ is a well-defined
homomorphism of chain complexes from $\red{A}\otimes_{U} B(U,S,\red{U})$ to
$\red{A}\otimes_{\red{U}} B(\red{U}, \red{S}, \red{U})$.
By Propositions~\ref{\labelauthor:pro:indbar} and~\ref{\labelauthor:pro:LU}, the complex
	$\red{A}\otimes_{\red{U} } B(\red{U}, \red{S}, \red{U})$ is an
	$(\red{A}, \red{S})$-projective resolution of $\red{A}$, and, in particular,
	it   is exact.
To finish the proof we will show that $\phi$ is an isomorphism, that is that for every
$k\ge -1$ the map $\phi_k$ is an isomorphism.

The map $\phi_{-1}$ is an isomorphism,
	since both $\red{A}$ and $\red{U}$ get the structure of right and left
	$U$-modules, respectively, via the projection $U\to \red{U}$.

Suppose now that $k\ge 0$. Then
	\begin{align*}
		\red{A} \otimes_U U^{\otimes_S ^{(k+1)}} \otimes_S \red{U}& \cong
		\red{A} \otimes_S U^{\otimes_S^k} \otimes_S \red{U}\\
\red{A} \otimes_{\red{U}} {\red{U}}^{\otimes_{\red{S}}^ {(k+1)}} \otimes_{\red{S}} \red{U}& \cong
		\red{A} \otimes_{\red{S}} {\red{U}}^{\otimes_{\red{S}}^ k}
		\otimes_{\red{S}} \red{U}.
	\end{align*}
	Under these isomorphisms $\phi_k$ corresponds to
	\begin{equation*}
		\psi_k\colon \red{A} \otimes_S U^{\otimes_S^ k} \otimes_S \red{U}
\to \red{A} \otimes_{\red{S}} {\red{U}}^{\otimes_{\red{S}}^ k}
		\otimes_{\red{S}} \red{U}.
	\end{equation*}
We will write  for a moment $e_1 = e$, $e_2 = \ce$, $S_1 = e_1 S e_1$ and
$S_2 = e_2Se_2$. As $S $ is diagonal,   using repeatedly  Proposition~\ref{\labelauthor:pro:Mac}, we get
\begin{equation*}
	\red{A} \otimes_S U^{\otimes_S^k} \otimes_S \red{U} \cong
\!\!\!\!\!	\bigoplus_{(i_0,\dots,i_k)\in \left\{ 1,2 \right\}^{k+1}}\!\!\!\!\!
	\red{A}e_{i_0} \otimes_{S_{i_0}} e_{i_0} U e_{i_1} \otimes_{S_{i_1}}
	\dots \otimes_{S_{i_{k-1}}}
	e_{i_{k-1}} U e_{i_k} \otimes_{S_{i_k}} e_{i_k} \red{U}.
\end{equation*}
Since $\red{A} e_1 =0$ and $ e_2 U e_1 = 0$, one sees that the only
non-zero summand in the above direct sum corresponds to the multi-index
$(2,2,\dots,2)$. Therefore, as $\red{A}e_2= \red{A}$ and $ e_2\red{U} = \red{U}$, we have that
\begin{align*}
	\psi_k\colon \red{A} \otimes_{S_2} \left(\ce U
	\ce\right)^{\otimes_{S_2}^k}
	\otimes_{S_2} \red{ U }
 &\to \red{A} \otimes_{\red{S}} {\red{U}}^{\otimes_{\red{S}}^k}
		\otimes_{\red{S}} \red{U}\\[2ex]
		[a] \otimes u_1 \otimes \dots \otimes u_k \otimes [u] &\mapsto
		[a]\otimes [u_1]\otimes \dots \otimes [u_k] \otimes [u].
\end{align*}
But these maps are isomorphisms, since  $S_2 = \ce S \ce \cong
\red{S}$ and $\ce U \ce \cong \red{U}$,  by
Proposition~\ref{\labelauthor:pro:upper}.
\end{proof}
Let $\K$ be a field. Suppose that $A$ is a finite dimensional algebra over
$\K$ and
 $e\in A$ is an idempotent  such that
$AeA$ is a projective left or right $A$-module. Then,  as we mentioned in the introduction, $AeA$ is a stratifying
ideal. Next we give an example of a finite dimensional  $\K$-algebra $A$, with an idempotent
$e$, such that $AeA$ is not projective, although it is a stratifying ideal by Theorem~\ref{\labelauthor:teo:main}.

\begin{example} Given two rings $A$, $B$, and an  $A$-$B$-bimodule $M$, we have
a ring structure on $R:=A\oplus M \oplus B$, given by
\begin{align*}
	(a_1,m_1,b_1) (a_2,m_2,b_2) := (a_1a_2, a_1m_2 + m_1 b_2, b_1b_2).
\end{align*}
The ring $R$ is upper triangular with respect  the idempotent $(1_A,0,0)$,
\begin{equation*}
	R=	\left(
	\begin{array}{cc}
		A & M\\
		0 & B
	\end{array}
	\right),
\end{equation*}
and is a lower triangular ring
with respect the idempotent $(0,0,1_B)$,
\begin{equation*}
	R=	\left(
	\begin{array}{cc}
		B & 0\\
		M & A
	\end{array}
	\right).
\end{equation*}
Let $A := \K$, $B :=
\left.\raisebox{0.3ex}{$\K[x]$}\middle/\!\raisebox{-0.3ex}{$\left( x^2
\right)$}\right. $, $M
:=\left.\raisebox{0.3ex}{$B$}\middle/\!\raisebox{-0.3ex}{$xB$}\right.$. Note
that $M$ is a one-dimensional  $A$-$B$-bimodule. Let  $v$ be a  generator of
$M$.  Then we
get a lower triangular ring
\begin{equation*}
	L:= 	\left(
	\begin{array}{cc}
		\left.\raisebox{0.3ex}{$\K[x]$}\middle/\!\raisebox{-0.3ex}{$\left(
		x^2 \right)$}\right.
	& 0 \\[2ex]
	\left\langle v \right\rangle
	& \K
	\end{array}
	\right).
\end{equation*}
Taking
$A:=\left.\raisebox{0.3ex}{$\K[y]$}\middle/\!\raisebox{-0.3ex}{$\left( y^2
\right)$}\right. $, $M :=
\left.\raisebox{0.3ex}{$A$}\middle/\!\raisebox{-0.3ex}{$Ay$}\right.$,
$B:=\K$,
once more we have that $M$ is a one-dimensional $A$-$B$-bimodule. We denote its generator by
$w$.
We get an upper triangular ring
\begin{equation*}
	U: = \left(
\begin{array}{cc}
	\left.\raisebox{0.3ex}{$\K[y]$}\middle/\!\raisebox{-0.3ex}{$\left(
	y^2 \right)$}\right.
	&
\left\langle w \right\rangle	
\\[2ex] 0 & \K	
\end{array}
\right). \end{equation*}
Note that both $L$ and $U$ contain the semisimple subring
\begin{equation*}
S : =
\left(
\begin{array}{cc}
	\K & 0 \\[2ex] 0 & \K
\end{array}
\right).
\end{equation*}
Let $e=\left(\begin{smallmatrix} 1 & 0 \\ 0 & 0 \end{smallmatrix}\right)\in S$ and
	$\red{e} = 1_S -e = \left(\begin{smallmatrix} 0 & 0 \\ 0 & 1
	\end{smallmatrix}\right)$. Then
applying Proposition~\ref{\labelauthor:pro:Mac}, we get
\begin{equation*}
	L\otimes_S U =
	\left(
	\begin{array}{cc}
		eL\otimes_S U e & eL\otimes_S U \red{e}\\
		\red{e} L\otimes_S U e & \red{e} L \otimes_S U\red{e}
	\end{array}
	\right)=
	\left(
	\begin{array}{cc}
		\left\langle e_{11}, x, y, x y \right\rangle &
		\left\langle w, xw \right\rangle \\[2ex]
		\left\langle v, vy \right\rangle &
		\left\langle vw, e_{22} \right\rangle
	\end{array}
	\right).
\end{equation*}
In the above formula we omitted $\otimes$ between the elements in $L$ and
$U$, wrote
$e_{ii}$ for the products $1\otimes 1$ at position $(i,i)$, and
 abbreviated $x\otimes 1$ by $x$, $1\otimes y$ by $y$, $1\otimes w$ by
 $w$, $v\otimes 1$ by $v$.

 We will define a multiplication in  the vector space $A=L\otimes_S U$,
such that  $$l \otimes u= (l \otimes 1_U)\cdot (1_L \otimes u).$$ Then $A$ will be a twisted product of the subalgebras $L\otimes_S S \cong L$ and
$S\otimes_S U\cong U$. To define such product it is enough to define the images of the
elements of $U\otimes_S L$ under the map
\begin{equation*}
	\tau  \colon  U\otimes_S L \xrightarrow{\cong} S\otimes_S U \otimes_S L\otimes_S S
	\rightarrow A\otimes_S A \xrightarrow{\mu_A} A.
\end{equation*}
Moreover, the above map restricted to the subspaces $S\otimes_S L$ and $U\otimes_S
S$, should be $s\otimes l\mapsto sl\otimes 1_U$ and $u\otimes s\mapsto 1_L
\otimes us$.  Applying
Proposition~\ref{\labelauthor:pro:Mac}, we get
\begin{align*}
U\otimes_S L& =
\left(
\begin{array}{cc}
	eU\otimes_S Le & eU\otimes_S L \red{e}\\[2ex]
	\red{e}U\otimes_S L e & \red{e} U\otimes_S L \red{e}
\end{array}
\right) \\[3ex]&=
\left(
\begin{array}{cc}
\left\langle 1\otimes 1, y\otimes 1, 1\otimes x, y\otimes x,w\otimes v
\right\rangle
&  \left\langle w\otimes 1\right\rangle\\[2ex]
\left\langle 1\otimes v\right\rangle &\left\langle  1\otimes 1\right\rangle
\end{array}
\right).
\end{align*}
So we only need to know $\tau(y\otimes x)$ and $\tau(w\otimes v)$.
Define
\begin{align*}
\tau(y\otimes x)	 & = xy , & \tau(w\otimes v) &= 0.
\end{align*}
It is easy to check that the resulting multiplication in $A = L\otimes_S U$ is
associative. By construction $A$ is a twisted product of subalgebras $L\otimes_S
S$ and $S\otimes_S U$. Thus we can apply Theorem~\ref{\labelauthor:teo:main} to $A$, and we get
that $AeA$ is a stratifying ideal. By direct computation one verifies that
\begin{align*}
	\dim(Ae) & = 6, & \dim(A\red{e}) & = 4, & \dim(Ae A) &= 9.
\end{align*}
Since $e$ and $\red{e}$ are primitive idempotents, the modules $Ae$ and $A\red{e}$ are indecomposable projective. Therefore we see that
$Ae A$ is not projective, as $9$ can not be represented as an integral
combination of $6$ and $ 4$.
Thus $AeA$ is an example of a non-projective stratifying ideal.
\end{example}

\end{document}